\documentclass[12pt]{article}
\usepackage{amsmath, mathrsfs}
\usepackage{amsfonts}
\usepackage{amsthm}
\usepackage{graphicx}
\graphicspath{{Figures/}} 
\usepackage{overpic}
\usepackage{amssymb} 
\usepackage{pictex}
\usepackage{rotating}  
\usepackage{color}
\usepackage{cite} 

\usepackage{enumitem} 

\usepackage{accents} 

\usepackage{yhmath} 
\usepackage{etex} 

\usepackage{comment}

\numberwithin{equation}{section}

\newtheorem{theorem}{Theorem}[section]
\newtheorem*{theorem*}{Theorem}
\newtheorem{proposition}[theorem]{Proposition}
\newtheorem{lemma}[theorem]{Lemma}

\newtheorem{Definition}[theorem]{Definition}
\newenvironment{definition}{\begin{Definition}\rm}{\end{Definition}}
\newtheorem{Remark}[theorem]{Remark}

\newtheorem{RHproblem}[theorem]{RH problem}

\newtheorem{Example}[theorem]{Example}

\newcommand{\C}{\mathbb{C}}

\newcommand{\N}{\mathbb{N}}

\newcommand{\R}{\mathbb{R}}

\usepackage{color}

\def\det{\mathop{\mathrm{det}}\nolimits}

\renewcommand{\tilde}{\widetilde}

\usepackage[bookmarksopen, naturalnames]{hyperref}
\begin{document}
\title{Asymptotic Approximate Fekete Arrays}
\author{T. Bloom, L. Bos and N. Levenberg}
\date{}

\maketitle

\begin{abstract}
The notion of asymptotic Fekete arrays, arrays of points in a compact set $K\subset \C^d$ which behave asymptotically like Fekete arrays, has been well-studied, albeit much more recently in dimensions $d>1$. Here we show that one can allow a more flexible definition where the points in the array need not lie in $K$. Our results, which work in the general setting of weighted pluripotential theory, rely heavily, in the multidimensional setting, on the ground-breaking work of Berman, Boucksom and Nystrom from \cite{BBN}.
\end{abstract}

\bigskip

\section{Introduction}  Let $K$ be a compact set in $\C$. For $n=1,2,...$ 
$$\delta_n(K):=\max_{z_0,...,z_n\in K}\prod_{j<k}|z_j-z_k|^{1/{n +1\choose 2}}$$
is called the $n-th$ order diameter of $K$. Note that $$VDM(z_0,...,z_n)=\det [z_i^{j}]_{i,j=0,1,...,n}=\prod_{j<k}(z_j-z_k)$$
$$=\det\left[
\begin{array}{ccccc}
 1 &z_0 &\ldots  &z_0^n\\
  \vdots  & \vdots & \ddots  & \vdots \\
1 &z_n &\ldots  &z_n^n
\end{array}
\right] $$
is a classical Vandermonde determinant; the basis monomials $1,z,...,z^n$ for the space of polynomials of degree at most $n$ are evaluated at the points $z_0,...,z_n$. The sequence of numbers $\{\delta_n(K)\}$ is decreasing and hence the limit 
	\begin{equation} \lim_{n\to \infty} \bigl[\max_{\lambda_i\in K}|VDM(\lambda_0,...,\lambda_n)\bigr] ^{1/{n+1 \choose 2}}:=\delta(K) \end{equation} 
	exists and is called the transfinite diameter of $K$. Points  $z_{n0},...,z_{nn}\in K$ for which $$|VDM(z_{n0},...z_{nn})|= |\det
\left[
\begin{array}{ccccc}
 1 &z_{n0} &\ldots  &z_{n0}^n\\
  \vdots  & \vdots & \ddots  & \vdots \\
1 &z_{nn} &\ldots  &z_{nn}^n
\end{array}
\right]| $$
 is maximal are called Fekete points of order $n$ for $K$. For such Fekete arrays $\{z_{nj}\}_{j=0,...,n; \ n=1,2,...}$, it is classical that if $K$ is not polar, then 
 \begin{equation}\label{feklim} \mu_n:=\frac{1}{n+1}\sum_{j=0}^n \delta_{z_{nj}}\to \mu_K \ \hbox{weak-*}\end{equation}
where $\mu_K$ is the equilibrium measure for $K$; i.e., the probability measure on $K$ of minimal logarithmic energy. In fact, for {\it asymptotically Fekete arrays} $\{z_{nj}\}_{j=0,...,n; \ n=1,2,...}\subset K$, i.e., such that
\begin{equation} \label{avdm} \lim_{n\to \infty} |VDM(z_{n0},...,z_{nn})|^{1/{n+1 \choose 2}}=\delta(K), \end{equation} 
(\ref{feklim}) holds. 

There are higher dimensional analogues of asymptotically Fekete arrays and equilibrium measures associated to nonpluripolar compact sets $K$ in $\C^d, \ d>1$; cf., \cite{zah} for the definition and properties of transfinite diameter in this setting. A striking generalization of (\ref{feklim}) was achieved by Berman, Boucksom and Nystrom in \cite{BBN}. There the authors proved a much more general result, one which, in the setting of the current work, requires {\it weighted} pluripotential theory. In both the univariate and multivariate cases, one regularly encounters situations where ``near extremal'' arrays contain points lying {\it outside} of $K$; cf., \cite{BE} and \cite{Pr}. It is the purpose of this note to show that, by suitably modifying the arguments in \cite{BBN}, one still recovers the appropriate generalization of (\ref{feklim}), even in the weighted setting.

In the next section, we repeat a standard argument from \cite{blchrist} to deal with the weighted, univariate case. Section 3 gives the necessary definitions and background in weighted pluripotential theory, and section 4 follows the strategy in \cite{BBN} to prove the main result on asymptotic, approximate weighted Fekete arrays in $\C^d, \ d>1$. Using the notation and terminology from those sections, here is the statement:

\begin{theorem} \label{maint} Let $K\subset \C^d$ be compact and nonpluripolar, let $\{K_n\}$ be a sequence of compact sets which decrease to $K$; i.e., $K_{n+1}\subset K_n$ for all $K$ and 
$K =\cap_n K_n$, and let $w$ be an admissible weight function on $K_n$ for $n\geq n_0$ whose restriction to $K$ is admissible. For each $n$, take points $x_1^{(n)},x_2^{(n)},\cdots,x_{m_n}^{(n)}\in K_n$ for which 
\begin{equation}\label{wam}
 \lim_{n\to \infty}\bigl[|VDM(x_1^{(n)},\cdots,x_{m_n}^{(n)})|w(x_1^{(n)})^nw(x_2^{(n)})^n\cdots w(x_{m_n}^{(n)})^n\bigr]^{\frac{1}{l_n}}=\delta^w(K)
\end{equation}
and let $\mu_n:= \frac{1}{m_n}\sum_{j=1}^{m_n} \delta_{x_j^{(n)}}$. Then
$$
\mu_n \to \mu_{K,Q}:=(dd^cV_{K,Q}^*)^d \ \hbox{weak}-*.
$$
\end{theorem}

\noindent In fact this is a special case of a general result on a sequence $\{\mu_n\}$ of measures, $\mu_n$ supported on $K_n$, given as Theorem \ref{genres} in section 4.
\medskip

\noindent{\bf Acknowledgement.} We thank Jean-Paul Calvi for an observation which greatly simplified the proof of Proposition \ref{DECR}.

\section{The univariate case} Let $K\subset \C$ be compact and let ${\mathcal M}(K)$ denote the convex set of probability  measures on $K$. For $\mu \in {\mathcal M}(K)$ define the logarithmic energy 
$$I(\mu):=\int_K \int_K \log {1\over |z-\zeta|}d\mu(z) d\mu(\zeta).$$
Either $\inf_{\mu \in {\mathcal M}(K)} I(\mu)=:I(\mu_{K})<+\infty$ for a unique $\mu_K\in {\mathcal M}(K)$ or else $I(\mu)=+\infty$ for all $\mu \in {\mathcal M}(K)$; this latter occurs when $K$ is polar. Suppose $K$ is nonpolar. Let $\{\epsilon_n\}$ be a decreasing sequence of positive numbers with $\lim_{n\to \infty} \epsilon_n =0$; and set
$$K_{n}:=\{z\in \C^d: \hbox{dist}(z,K)\leq \epsilon_n\}.$$ These are compact sets decreasing to $K$ (in fact, regular compacta; cf., \cite{klimek}); thus 
$$\lim_{n\to \infty} \delta(K_n)=\delta(K). $$
Given an array of points $\{z_{nj}\}_{j=0,...,n; \ n=1,2,...}$ where $z_{n0},...,z_{nn}\in K_{n}$, it follows from modifying the standard proof for asymptotic Fekete arrays in $K$ that under the condition
$$\lim_{n\to \infty} |VDM(z_{n0},...z_{nn})|^{\frac{1}{\binom{n+1}{2}}}=\delta(K),$$
we have
$$\mu_n:=\frac{1}{n+1}\sum_{j=0}^n \delta_{z_{nj}}\to \mu_K \ \hbox{weak-*}.$$
The idea, utilized in \cite{BBCL}, is simply to show that, if one replaces $\mu_n$ by $\tilde \mu_n$ by spreading the point masses to little disks or circles centered at these points with radii $r_n\to 0$ where $r_n\leq \epsilon_n$, then the hypothesis on the $VDM$'s gives that $\tilde \mu_n \to \mu_K$ weak-* (any weak-* limit of $\{\tilde \mu_n\}$ has logarithmic energy equal to that of $\mu_K$, hence it must equal $\mu_K$). Any weak-* limit of $\{\mu_n\}$ is supported on $K$ and coincides with the weak-* limit of $\{\tilde \mu_n\}$. In particular, no condition on $\epsilon_n$ other than  $\lim_{n\to \infty} \epsilon_n =0$ is necessary. The same conclusion holds for any sequence $\{K_n\}$ of compact sets which decrease to $K$; i.e., $K_{n+1}\subset K_n$ for all $K$ and 
$K =\cap_n K_n$, as follows from Proposition \ref{wtdaf} below.

A similar result holds for a weighted situation where  $w$ is an admissible weight function on $K_{n}$ for $n\geq n_0$ whose restriction to $K$ is admissible. We refer the reader to \cite{safftotik} for details of this theory. Let $K\subset \C$ be compact and let $w$ be an admissible weight function on $K$:  $w$ is a nonnegative, uppersemicontinuous function with
$\{z\in K:w(z)>0\}$ nonpolar -- hence $K$ is not polar. We write $Q:=-\log w$. Associated to $K,Q$ is a {\it weighted energy minimization problem}: for a probability measure $\tau $ on $K$, we consider the weighted energy
	$$I^w(\tau):=\int_K \int_K \log \frac {1}{|z-t|w(z)w(t)}d\tau(t)d\tau(z)=I(\tau) +2\int_KQd\tau$$ and  find $\inf_{\tau}I^w(\tau)$ where the infimum is taken over all probability measures $\tau $ with compact support in $K$. There exists a unique  minimizer which we denote by $\mu_{K,Q}$. The associated discrete problem leads to the {\it weighted transfinite diameter of $K$ with respect to $w$}:
	\begin{equation} \label{wtdvdm} \delta^w(K):=\lim_{n\to \infty} \bigl[\max_{\lambda_i\in K}|VDM(\lambda_0,...,\lambda_n)|w(\lambda_0)^n\cdots w(\lambda_n)^n\bigr] ^{1/{n+1\choose 2}}:=\lim_{n\to \infty}\delta_n^w(K).\end{equation} 
	 We have (cf., \cite{safftotik} )
 \begin{equation} \label{wtdenmin}I^w(\mu_{K,Q})= \inf_{\tau  \in {\mathcal M}(K)}I^w(\tau)=-\log \delta^w(K). \end{equation}

\begin{proposition} \label{wtdaf} Let $K\subset \C$ be compact and nonpluripolar and let $\{K_n\}$ be any sequence of compact sets which decrease to $K$; i.e., $K_{n+1}\subset K_n$ for all $K$ and 
$K =\cap_n K_n$. Let $w=e^{-Q}$ be any admissible weight function on $K_n$ for $n\geq n_0$ whose restriction to $K$ is admissible. Given an array of points $\{z_{nj}\}_{j=0,...,n; \ n=1,2,...}$ where $z_{n0},...,z_{nn}\in K_n$, if
$$\lim_{n\to \infty} [|VDM(z_{n0},...z_{nn})|w(z_{n0})^n\cdots w(z_{nn})^n]^{\frac{1}{\binom{n+1}{2}}}=\delta^w(K),$$
we have
$$\mu_n:=\frac{1}{n+1}\sum_{j=0}^n \delta_{z_{nj}}\to \mu_{K,Q} \ \hbox{weak-*}.$$

\end{proposition}

\begin{proof} We follow the argument in \cite{blchrist}. Since $K_n\searrow K$, any weak-* limit $\mu$ of $\{\mu_n\}$ will be a probability measure supported on $K$. Take a subsequence $\{\mu_{n_j}\}$ of $\{\mu_n\}$ which converges to $\mu$. It suffices to show $I^w(\mu)=-\log \delta^w(K)$ since $I^w(\mu_{K,Q})= -\log \delta^w(K)$ and the minimizer $\mu_{K,Q}$ is unique. 

We take continuous weight functions $\{w_m\}$ on $K_{n_0}$ with $w_m \searrow w$ and $w_m\geq a_m >0$. For $M\in \R$ let
$$h_{M,m}(z,t):=\min[M,\log \frac{1}{|z-t|w_m(z)w_m(t)}]\leq \log \frac{1}{|z-t|w_m(z)w_m(t)} \ \hbox{and}$$
$$h_M(z,t):=\min[M,\log \frac{1}{|z-t|w(z)w(t)}]\leq \log \frac{1}{|z-t|w(z)w(t)}.$$
Then $h_{M,m}\leq h_M$. Every continuous function $F(z,t)$ on $K_{n_0}\times K_{n_0}$ can be uniformly approximated by finite sums of the form $\sum_j f_j(z)g_j(t)$ where $f_j,g_j$ are continuous on $K_{n_0}$. Thus $\mu_{n_j}\times \mu_{n_j}\to \mu \times \mu$ and hence
$$I^w(\mu)=\lim_{M\to \infty}\lim_{m\to \infty} \int_K \int_K h_{M,m}(z,t)d\mu(z)d\mu(t)$$
$$=\lim_{M\to \infty}\lim_{m\to \infty} \int_{K_{n_0}} \int_{K_{n_0}} h_{M,m}(z,t)d\mu(z)d\mu(t)$$
$$=\lim_{M\to \infty}\lim_{m\to \infty} \lim_{j\to \infty} \int_{K_{n_0}} \int_{K_{n_0}} h_{M,m}(z,t)d\mu_{n_j}(z)d\mu_{n_j}(t)$$
$$\leq \lim_{M\to \infty}\limsup_{j\to \infty} \int_{K_{n_0}} \int_{K_{n_0}} h_M(z,t)d\mu_{n_j}(z)d\mu_{n_j}(t),$$
the last inequality arising since $h_{M,m}\leq h_M$.

For convenience in notation, we write $z^{(n_j)}_k:=z_{n_j k}$. Then for $k\not = l$,
$$h_M(z^{(n_j)}_k,z^{(n_j)}_l)\leq \log \frac{1}{|z^{(n_j)}_k-z^{(n_j)}_l|w(z^{(n_j)}_k)w(z^{(n_j)}_l)}$$
so that
$$\int_{K_{n_0}} h_M(z,t)d\mu_{n_j}(z)d\mu_{n_j}(t)\leq \frac{1}{n_j}M +(\frac{1}{n_j^2-n_j})\bigl[\sum_{k\not = l}\log \frac{1}{|z^{(n_j)}_k-z^{(n_j)}_l|w(z^{(n_j)}_k)w(z^{(n_j)}_l)}\bigr].$$
By hypothesis, given $\epsilon >0$, 
$$(\frac{1}{n_j^2-n_j})\bigl[\sum_{k\not = l}\log \frac{1}{|z^{(n_j)}_k-z^{(n_j)}_l|w(z^{(n_j)}_k)w(z^{(n_j)}_l)}\bigr]\leq -\log[\delta^w(K)-\epsilon]$$
for $n_j\geq n_j(\epsilon)$. Thus we can assume $w(z^{(n_j)}_k)>0$ and hence
$$I^w(\mu)\leq \lim_{M\to \infty}\limsup_{j\to \infty}\frac{1}{n_j}M - \log[\delta^w(K)-\epsilon]=-\log[\delta^w(K)-\epsilon].$$
This holds for all $\epsilon >0$ and hence $I^w(\mu)\leq -\log \delta^w(K)$. Since $\mu\in \mathcal M(K)$ implies  $I^w(\mu)\geq -\log \delta^w(K)$, equality holds.

\end{proof}

We will call arrays as in Proposition \ref{wtdaf} {\it asymptotic approximate (weighted) Fekete arrays} (AAF or AAWF for short). 

\section{Weighted pluripotential theory in $\C^d, \ d>1$} Let $e_1(z),...,e_j(z),...$ be a listing of the monomials
$\{e_i(z)=z^{\alpha(i)}=z_1^{\alpha_1}\cdots z_d^{\alpha_d}\}$ in
$\C^d$ indexed using a lexicographic ordering on the multiindices $\alpha=\alpha(i)=(\alpha_1,...,\alpha_d)\in \N^d$, but with deg$e_i=|\alpha(i)|$ nondecreasing. We write $|\alpha|:=\sum_{j=1}^d\alpha_j$. For 
$\zeta_1,...,\zeta_m\in \C^d$, let
\begin{equation} \label{vdmcn}VDM(\zeta_1,...,\zeta_m)=\det [e_i(\zeta_j)]_{i,j=1,...,m}  \end{equation}
$$= \det
\left[
\begin{array}{ccccc}
 e_1(\zeta_1) &e_1(\zeta_2) &\ldots  &e_1(\zeta_m)\\
  \vdots  & \vdots & \ddots  & \vdots \\
e_m(\zeta_1) &e_m(\zeta_2) &\ldots  &e_m(\zeta_m)
\end{array}
\right]$$
be a generalized Vandermonde determinant. In analogy with the univariate case, 
for a compact subset $K\subset \C^d$ let
$$V_m =V_m(K):=\max_{\zeta_1,...,\zeta_m\in K}|VDM(\zeta_1,...,\zeta_m)|.$$
Let $m_n$ be the number of monomials $e_i(z)$ of degree at most $n$ in $d$ variables, i.e., the dimension of $\mathcal P_n$, the space of holomorphic polynomials of degree at most $n$, and let $l_n:=\sum_{j=1}^{m_n}\hbox{deg} e_j$. Define $\delta_n(K) :=V_{m_n}^{1/l_n}$. Zaharjuta \cite{zah} showed that the limit 
\begin{equation} \label{tdlim}\delta(K):=\lim_{n\to \infty} \delta_n(K) \end{equation} exists; this is the {\it transfinite diameter} of $K$. We remark that $ l_n= \frac{d}{d+1}nm_n$.

The same definition of admissible weight function $w$ is used for $K\subset \C^d$ compact (of course now $\{z\in K:w(z)>0\}$ should be nonpluripolar). For $K$ compact, let 
$w=e^{-Q}$ be an admissible weight function on
$K$.  Given $\zeta_1,...,\zeta_{m_n}\in K$, let
$$W(\zeta_1,...,\zeta_{m_n}):=VDM(\zeta_1,...,\zeta_{m_n})w(\zeta_1)^{n}\cdots w(\zeta_{m_n})^{n}$$
$$= \det
\left[
\begin{array}{ccccc}
 e_1(\zeta_1) &e_1(\zeta_2) &\ldots  &e_1(\zeta_{m_n})\\
  \vdots  & \vdots & \ddots  & \vdots \\
e_{m_d}(\zeta_1) &e_{m_d}(\zeta_2) &\ldots  &e_{m_n}(\zeta_{m_n})
\end{array}
\right]\cdot w(\zeta_1)^{n}\cdots w(\zeta_{m_n})^{n}$$
be a {\it weighted Vandermonde determinant}. Define an {\it $n-$th order weighted Fekete set for $K$ and $w$} to be a set of $m_n$ points $\zeta_1,...,\zeta_{m_n}\in K$ with the property that
$$W_{m_n}=W_{m_n}(K):=|W(\zeta_1,...,\zeta_{m_n})|=\sup_{\xi_1,...,\xi_{m_n}\in K}|W(\xi_1,...,\xi_{m_n})|.$$
In analogy with the univariate notation, we also set $\delta^w_n(K):= W_{m_n}^{1/ l_n}$. Then the limit
\begin{equation} \label{deltaw} \delta^w(K):=\lim_{n\to \infty}\delta^w_n(K)\end{equation}
exists \cite{[BL]}; this is the weighted transfinite diameter of $K$ with respect to $w$. 

We define the weighted extremal function or weighted pluricomplex Green function $V^*_{K,Q}(z):=\limsup_{\zeta \to z}V_{K,Q}(\zeta)$ where
$$V_{K,Q}(z):=\sup \{u(z):u\in L(\C^d), \ u\leq Q \ \hbox{on} \ K\} .$$
Here, $L(\C^d):=\{u\in PSH(\C^d): u(z)-\log |z| =0(1), \ |z|\to \infty\}$ are the psh functions in $\C^d$ of minimal growth. We have $V^*_{K,Q}\in L(\C^d)$ and $\mu_{K,Q}:=(dd^cV_{K,Q}^*)^d$ is a well-defined positive measure with support in $K$. If $w\equiv 1$; i.e., $Q=-\log w\equiv 0$, we simply write $V_K, V_K^*$ and $\mu_K:= (dd^cV_{K}^*)^d$. In this setting, we say $K$ is {\it regular} if $V_K$ is continuous. Here we are normalizing our definition of $dd^c$ so that $\mu_{K,Q}, \mu_K$ are probability measures.

\section{AAF and AAWF in $\C^d, \ d>1$} The analogue of Proposition \ref{wtdaf} in $\C^d, \ d>1$ holds but the proof is much more difficult. For $E\subset \C^d$, a measure $\nu$ on $E$, and a weight $w$ on $E$, we use the notation
\begin{equation}\label{M}
G_n^{\nu,w}:=\left[\int_E\overline{e_i(z)}e_j(z)w(z)^{2n}d\nu\right]\in\C^{m_n\times m_n}
\end{equation}
for the weighted Gram matrix of $\nu$ of order $n$ with respect to the standard basis monomials $\{e_1,...,e_{m_n}\}$ in $\mathcal P_n$. We let
$$Z_n:= \int_E \cdots \int_E |VDM(z_1,...,z_{m_n})|^2 w(z_1)^{2n} \cdots w(z_{m_n})^{2n}d\nu(z_1) \cdots d\nu(z_{m_n})$$
and we have
\begin{equation}
\label{nthberg}
B_n^{\nu,w}(z):=\sum_{j=1}^{m_n} |q_j^{(n)}(z)|^2w(z)^{2n},
\end{equation}
the $n-th$ {\it Bergman function} of $E,w,\nu$. Here, $\{q_j^{(n)}\}_{j=1,...,m_n}$ is an orthonormal basis for $\mathcal P_n$ with respect to the weighted $L^2-$norm $p\to ||w^np||_{L^2(\nu)}$. The following calculations are straightforward.

\begin{lemma}\label{DetasInt}
Suppose that $\nu\in{\mathcal M}(E)$ and that $w$ is an admissible weight on $E$. Then
\begin{equation}\label{gneqn}
{\rm det}(G_n^{\nu,w})= \end{equation}
$$\frac{1}{m_n!}\int_{E^{m_n}}|VDM(z_1,\cdots,z_{m_n})|^2 w(z_1)^{2n}\cdots w(z_{m_n})^{2n} d\nu(z_1)\cdots d\nu(z_{m_n})=\frac{Z_n}{m_n!}$$
and
\begin{equation}\label{bneqn}
B_n^{\nu,w}(z)=
\end{equation}
$$\frac{m_n}{Z_n}\int_{E^{m_n-1}}
|VDM(z,z_2,\cdots,z_{m_n})|^2 w(z)^{2n}w(z_2)^{2n}\cdots w(z_{m_n})^{2n}d\nu(z_2)\cdots d\nu(z_{m_n}).$$

\end{lemma}


The notion of {\it optimal measure} will be useful; see \cite{BBLW} for more information.

\begin{definition}\label{optmeas} If a probability measure $\mu$ on $E$ has the property that
\begin{equation}\label{wa} {\rm det}(G_n^{\mu',w})\le {\rm det}(G_n^{\mu,w})
\end{equation}
for all other probability measures $\mu'$ on $E$ then $\mu$ is said to be an {\it optimal measure} of
order $n$ for $E$ and $w.$\end{definition}

Let $K$ be a nonpluripolar compact set in $\C^d$, and let $\{K_n\}$ be any sequence of compact sets which decrease to $K$; i.e., $K_{n+1}\subset K_n$ for all $K$ and 
$K =\cap_n K_n$. Let $w$ be an admissible weight function on $K_{n}$ for $n\geq n_0$ whose restriction to $K$ is admissible. In this setting, an elementary but crucial result is a modification of Proposition 2.10 of \cite{lev}. 

\begin{proposition}\label{tfd} Suppose that, for the diagonal sequence $\{\delta_n^w(K_{n})\}_n$,
\begin{equation}\label{diag} \lim_{n\to\infty} \delta_n^w(K_{n})=\delta^w(K).\end{equation}
For $n=1,2,...$, let $\mu_n$ be an optimal measure of order $n$ for $K_{n}$ and $w.$ Then
\[\lim_{n\to\infty}{\rm det}(G_n^{\mu_n,w})^{\frac{1}{2l_n}}=\delta^w(K).\]
\end{proposition}
\begin{proof} Since $\mu_n$ is a probability measure, it follows from (\ref{gneqn}) that
\begin{equation}\label{4.7}
{\rm det}(G_n^{\mu_n,w})\le {1\over m_n!}(\delta_n^w(K_n))^{2l_n}.
\end{equation}
Now if $f_1,f_2,\cdots,f_{m_n}\in K_n$ are weighted Fekete points of order $n$ for $K_n$ and $w$, 
i.e., points in $K_n$ for which 
\[|VDM(z_1,\cdots,z_{m_n})|w^n(z_1)w^n(z_2)\cdots w^n(z_{m_n})\]
is maximal, then the discrete measure 
\begin{equation} \label{48}
\nu_n={1\over {m_n}}\sum_{k=1}^{m_n} \delta_{f_k}
\end{equation}
is a candidate for an optimal measure of order $n$ for $K_n$ and $w$; hence
\[ {\rm det}(G_n^{\nu_n,w})\le {\rm det}(G_n^{\mu_n,w}).\]
But from (\ref{48})
\begin{eqnarray*}
{\rm det}(G_n^{\nu_n,w})&=&{1\over {m_n}^{m_n}}|VDM(f_1,\cdots,f_{m_n})|^2w(f_1)^{2n}w(f_2)^{2n}\cdots w(f_{m_n})^{2n} \\
&=&\left(\max_{z_i\in K}|VDM(z_1,\cdots,z_{m_n})|w^n(z_1)w^n(z_2)\cdots w^n(z_{m_n})\right)^2 \\
&=&{1\over {m_n}^{m_n}}(\delta_n^w(K_{n}))^{2l_n}
\end{eqnarray*}
so that 
\[{1\over {m_n}^{m_n}}(\delta_n^w(K_{n}))^{2l_n} \le {\rm det}(G_n^{\mu_n,w}) \]
and the result follows from this, (\ref{4.7}), and the hypothesis (\ref{diag}).
\end{proof}

The bulk of the proof of the analogue of Proposition \ref{wtdaf} in $\C^d, \ d>1$ is to modify the arguments in \cite{BBN} to show that if (\ref{diag}) holds; i.e., for the diagonal sequence,
$$ \lim_{n\to\infty} \delta_n^w(K_{n})=\delta^w(K),$$ 
then for $\{\mu_n\}$ a sequence of probability measures on $K_{n}$ with the property that 
\[\lim_{n\to\infty}{\rm det}(G_n^{\mu_n,w})^{\frac{1}{2l_n}}=\delta^w(K),\]
we have $\frac{1}{m_n}B_n^{\mu_n,w} d\mu_n \to \mu_{K,Q}$ (Theorem \ref{genres}).

We first show that (\ref{diag}) holds in a very general setting, beginning with the unweighted case.

\begin{proposition} \label{DECR} Let $K\subset \C^d$ be compact and nonpluripolar and let $\{K_n\}$ be any sequence of compact sets which decrease to $K$; i.e., $K_{n+1}\subset K_n$ for all $K$ and 
$K =\cap_n K_n$. Then 
$$\lim_{n\to \infty} \delta_n(K_n)=\delta(K).$$

\end{proposition}

\begin{proof} It is standard that $\delta$ is continuous under decreasing limits; i.e., 
$$\lim_{n\to \infty} \delta(K_n)=\delta(K);$$
and, by Zaharjuta \cite{zah}, for each compact set $E$,
$$\lim_{n\to \infty} \delta_n(E)=\delta(E).$$
We will use these facts.

First, for each $n$, $K\subset K_n$ so that $\delta_n(K) \leq \delta_n (K_n)$. Thus
$$\delta(K)=\liminf_{n\to \infty} \delta_n(K) \leq \liminf_{n\to \infty} \delta_n (K_n).$$
On the other hand, fixing $n_0$, for all $n>n_0$ we have $K_n\subset K_{n_0}$ so that $\delta_n(K_n) \leq \delta_n (K_{n_0})$. Thus 
$$\limsup_{n\to \infty} \delta_n(K_n) \leq \limsup_{n\to \infty} \delta_n (K_{n_0})=\delta(K_{n_0}).$$
This holds for each $n_0$; hence 
$$\limsup_{n\to \infty} \delta_n(K_n) \leq \lim_{n_0\to \infty} \delta(K_{n_0})=\delta(K).$$
\end{proof}

We turn to the weighted case, which requires a few more ingredients.

 \begin{proposition} \label{wDECR} Let $K\subset \C^d$ be compact and nonpluripolar and let $\{K_n\}$ be any sequence of compact sets which decrease to $K$; i.e., $K_{n+1}\subset K_n$ for all $K$ and 
$K =\cap_n K_n$. Let $w=e^{-Q}$ be any admissible weight function on $K_n$ for $n\geq n_0$ whose restriction to $K$ is admissible. Then 
$$\lim_{n\to \infty} \delta^w_n(K_n)=\delta(K).$$

\end{proposition} 

\begin{proof} We claim that
\begin{equation}\label{wkn} \lim_{n\to \infty} \delta^w(K_n)=\delta^w(K).\end{equation}
Since, as in the unweighted case, for each compact set $E$ and admissible weight $w$ on $E$
$$\lim_{n\to \infty} \delta^w_n(E)=\delta^w(E),$$
given (\ref{wkn}), we can repeat the proof of Proposition \ref{DECR} in this weighted setting.

To verify (\ref{wkn}), we first observe that since $K\subset K_n$, $\delta^w(K_n)\geq \delta^w(K)$ so that
$$\liminf_{n\to \infty} \delta^w(K_n)\geq \delta^w(K).$$
For the reverse inequality with limsup, we note that one can define 
a slightly different weighted transfinite diameter, in the spirit of Zaharjuta, via
$$d^w(K):=\exp\bigl[\frac{1}{|\Sigma|}\int_{\Sigma^0}\log {\tau^w(K,\theta)}d\theta \bigr] $$
(cf., \cite{[BL]} for the appropriate definitions and results).	There is a relationship between $\delta^w(K)$ and $d^w(K)$:
	\begin{equation}\label{reln} \delta^w(K)=\bigl(\exp [-\int_KQ(dd^cV_{K,Q}^*)^d]\bigr)^{1/d} d^w(K).\end{equation}
	Now it is straightforward that for {\it any} decreasing family of compact sets $\{K_n\}$ decreasing to $K$ and $w=e^{-Q}$ any admissible weight function on $K_n$ for $n\geq n_0$ whose restriction to $K$ is admissible, we have 
	$$ \lim_{n\to \infty} \tau^w(K_n,\theta) = \tau^w(K,\theta)$$
	for $\theta \in \Sigma^0$ and hence
	$$\lim_{n\to \infty} d^w(K_n)=d^w(K).$$
	But we also have $V_{K_n,Q}^*\nearrow V_{K,Q}^*$ pointwise on $\C^d$ except perhaps a pluripolar set so that $(dd^cV_{K_n,Q}^*)^d\to (dd^cV_{K,Q}^*)^d$ as measures. Since $Q$ is lowersemicontinuous,
	$$\liminf_{n\to \infty} \int_{K_n} Q(dd^cV_{K_n,Q}^*)^d\geq \int_KQ(dd^cV_{K,Q}^*)^d. $$
	Hence
	$$\limsup_{n\to \infty} \delta^w(K_n)= \exp [-\liminf_{n\to \infty}\int_{K_n} Q(dd^cV_{K,Q}^*)^d]\bigr)^{1/d} d^w(K)$$
	$$\leq \bigl(\exp [-\int_KQ(dd^cV_{K,Q}^*)^d]\bigr)^{1/d} d^w(K)= \delta^w(K).$$
	This shows that 
	$$\lim_{n\to \infty} \int_{K_n} Q(dd^cV_{K_n,Q}^*)^d= \int_KQ(dd^cV_{K,Q}^*)^d$$
	so that, from (\ref{reln}), we have (\ref{wkn}).

\end{proof}

Given a compact set $K$, let $w$ be an admissible weight function on $K_{n}$ for $n\geq n_0$ whose restriction to $K$ is admissible. For a real-valued, continuous function $u$ on $K_{n_0}$, we consider the weight $w_t(z):=w(z)\exp(-tu(z)),$ $t\in\R,$ and we let $\{\mu_n\}$ be a sequence of probability measures with $\mu_n$ supported on $K_n$. Define
\begin{equation}\label{fn}
f_n(t):=-{1\over 2 l_n}\log\,{\rm det}(G_n^{\mu_n,w_t}).\end{equation}
Note that only the values of $u$ on $K_n$ are needed to define $G_n^{\mu_n,w_t}$ and hence $f_n(t)$. Also note that $f_n(0)= -{1\over 2 l_n}\log\,{\rm det}(G_n^{\mu_n,w})$. We have the following (see Lemma 6.4 in \cite{BBnew}).

\begin{lemma}\label{1stderiv}We have
\[ f_n'(t)={d+1\over dm_n}\int_{K_n }u(z)B_n^{\mu_n,w_t}(z)d\mu_n.\]
In particular, 
$$f_n'(0)={d+1\over dm_n}\int_{K_n} u(z)B_n^{\mu_n,w}(z)d\mu_n$$
and if $B_n^{\mu_n,w}=m_n$ a.e. $\mu_n$, 
\begin{equation} \label{need?} f_n'(0)= {d+1\over d}\int_{K_n} u(z)d\mu_n.
\end{equation}
\end{lemma}

Before we give the proof, an illustrative example can be given if $\mu_n:= {1\over m_n}\sum_{j=1}^{m_n} \delta_{x_j}$. It is easy to see that $B_n^{\mu_n,w}(x_j) = m_n$ for $j=1,...,m_n$ so 
$$\log{\rm det}(G_n^{\mu_n,w_t}) $$
$$=\log\bigl( |W(x_1,...,x_{m_n})|^2e^{-2ntu(x_1)}\cdots e^{-2ntu(x_{m_n})}\bigr)$$
implies
$${d\over dt}\log{\rm det}(G_n^{\mu_n,w_t})={d\over dt}\bigl(-2tn\sum_{j=1}^{m_n} u(x_j)\bigr)$$
$$=-2n \sum_{j=1}^{m_n} u(x_j)=-2nm_n \int_{K_n} u(z){1\over m_n}B_n^{\mu_n,w}(z)d\mu_n.$$
Recalling that $ l_n= \frac{d}{d+1}nm_n$, this gives
$$-{1\over 2 l_n}\log\,{\rm det}(G_n^{\mu_n,w_t})={d+1\over dm_n}\int_{K_n }u(z)B_n^{\mu_n,w_t}(z)d\mu_n.$$
In this case, ${d\over dt}\log{\rm det}(G_n^{\mu_n,w_t})$ is a constant, independent of $t$; hence 
${d^2\over dt^2}\log{\rm det}(G_n^{\mu_n,w_t})\equiv 0$ -- see Lemma \ref{2ndderiv}.

\begin{proof} The proof we offer here is modified from \cite{BBLopt} and is 
based on the integral formulas of Lemma \ref{DetasInt}.

By  (\ref{gneqn}) we may write
\[ f_n(t)=-{1\over 2 l_n}\log(F_n)+{1\over 2 l_n}\log(m_n!)\]
where  
\[F_n(t):=\int_{K_n^{m_n}}V\exp(-tU)d\mu\]
and
$$
V:=V(z_1,z_2,\cdots,z_{m_n})=|VDM(z_1,\cdots,z_{m_n})|^2w(z_1)^{2n}\cdots w(z_{m_n})^{2n},$$
$$U:=U(z_1,z_2,\cdots,z_{m_n})=2n(u(z_1)+\cdots+u(z_{m_n})),$$
$$d\mu:=d\mu_n(z_1)d\mu_n(z_2)\cdots d\mu_n(z_{m_n}).$$
Further, by (\ref{bneqn}) for $w=w_t$ and $\mu=\mu_n,$ we have
\[ B_n^{\mu_n,w_t}(z)\]
\[=\frac{m_n}{Z_n}\int_{K_n^{m_n-1}}
V(z,z_2,z_3,\cdots,z_{m_n})\exp(-tU)d\mu_n(z_2)\cdots d\mu_n(z_{m_n})\]
where 
\[Z_n=Z_n(t):=m_n!\,{\rm det}(G_n^{\mu_n,w_t})=\int_{K_n^{m_n}}V\,\exp(-tU) d\mu.\]
Note that $Z_n(t)=F_n(t).$
Now
\[f_n'(t)=-{1\over 2 l_n}\frac{F_n'(t)}{F_n(t)}\]
and we may compute
$$
F_n'(t)=\int_{K_n^{m_n}} V(-U)\exp(-tU)d\mu_n(z_1)\cdots d\mu_n(z_{m_n})\\
$$
$$= -2n\int_{K_n^{m_n}} (u(z_1)+\cdots+u(z_{m_n}))V\exp(-tU)d\mu_n(z_1)\cdots d\mu_n(z_{m_n}).
$$
Notice that the integrand is symmetric in the variables and hence we
may ``de-symmetrize'' to obtain
$$F_n'(t)$$
$$=-2nm_n\int_{K_n^{m_n}}u(z_1)V(z_1,\cdots,z_{m_n})\exp(-tU)d\mu_n(z_1)\cdots d\mu_n(z_{m_n})$$
so that, integrating in all but the $z_1$ variable, we obtain
\[F_n'(t)=-2nm_n\int_{K_n} u(z)B_n^{\mu_n,w_t}(z)\frac{Z_n}{n}d\mu_n(z).\]
Thus, using the fact that $Z_n(t)=F_n(t),$ we obtain
$$f_n'(t)=\frac{d+1}{dm_n}\int_{K_n} u(z)B_n^{\mu_n,w_t}(z)d\mu_n(z)$$
as claimed.
In particular, $$f_n'(0)={d+1\over dm_n}\int_{K_n} u(z)B_n^{\mu_n,w}(z)d\mu_n$$
and if $B_n^{\mu_n,w}=m_n$ a.e. $\mu_n$, we recover (\ref{need?}):
$$f_n'(0)= {d+1\over d}\int_{K_n} u(z)d\mu_n.$$
 \end{proof}

The next result was proved in a different way in \cite{BBN}, Lemma 2.2, and also
in \cite{BBLW}, Lemma 3.6. We follow \cite{BBLopt}.

\begin{lemma}\label{2ndderiv}
The functions $f_n(t)$ are concave.  
\end{lemma}
\begin{proof} We show
that $f_n''(t)\le0.$ With the notation used in the proof of Lemma \ref{1stderiv},
\[f_n''(t)={1\over 2 l_n}\frac{(F_n'(t))^2-F_n''(t)}{F_n^2(t)}\]
and
\begin{eqnarray*}
F_n'(t)&=& -\frac{1}{m_n!}\int_{K_n^{m_n}} UV\exp(-tU)d\mu,\\
F_n''(t)&=&\frac{1}{m_n!}\int_{K_n^{m_n}}U^2V\exp(-tU)d\mu.
\end{eqnarray*}
We must show that $(F_n'(t))^2-F_n''(t)\le0.$ Now, for a fixed $t,$ we may
mulitply $V$ by a constant so that
\[\int_{K_n^{m_n}} V\exp(-tU)d\mu=1.\]
Let $d\gamma :=V\exp(-tU)d\mu.$ Then by the above formulas for $F_n'$ and $F_n'',$ we must show
that
\[\int_{K_n^{m_n}}U^2d\gamma \ge \left(\int_{K_n^{m_n}}U\,d\gamma\right)^2,\]
but this is a simple consequence of the Cauchy-Schwarz inequality. \end{proof}

Define 
$$g(t)=-\log(\delta^{w_t}(K))$$
so that $g(0)=-\log(\delta^{w}(K))$.
From the Berman-Boucksom differentiability result in \cite{BBnew} and their Rumely-type formula (cf. Theorem 5.1 in \cite{lev}), we have  
\begin{equation} \label{gp0} g'(0)= {d+1\over d}\int_Ku(z)(dd^cV_{K,Q}^*)^d \end{equation}
(cf., p. 61 of \cite{lev}). Note that for each $n$, $\mu_n$ is a candidate to be an optimal measure of order $n$ for $K_n$ and $w_t$. Thus, if $\mu_n^t$ is an optimal measure of order $n$ for $K_n$ and $w_t$, we have 
$$\det G_n^{\mu_n,w_t} \leq \det G_n^{\mu_n^t,w_t}$$
and, from Proposition \ref{tfd}, using (\ref{diag}) for the weight $w_t$, 
$$\lim_{n\to \infty}  \frac{1}{2l_n}\cdot \log \det G_n^{\mu_n^t,w_t}= \log \delta^{w_t}(K).$$
Thus with 
$$f_n(t):=-\frac{1}{2l_n}\log\,{\rm det}(G_n^{\mu_n,w_t})$$
as in (\ref{fn}),
\begin{equation}\label{fnp1} \liminf f_n(t) \geq g(t) \ \hbox{for all} \ t.\end{equation}
From Lemma \ref{1stderiv}, we have
\begin{equation}\label{fnp0} f_n'(0)={d+1\over dm_n}\int_{K_n} u(z)B_n^{\mu_n,w}(z)d\mu_n.\end{equation}

We state a calculus lemma, Lemma 3.1 from \cite{BBN}.

\begin{lemma}\label{calc}
Let $f_n$ be a sequence of concave functions on $\R$ and let $g$ be a function on $\R$. Suppose
$$\liminf f_n(t) \geq g(t) \ \hbox{for all} \ t \ \hbox{and} \ \lim f_n(0) = g(0)$$
and that $f_n$ and $g$ are differentiable at $0$. Then $\lim f_n'(0) = g'(0)$.
\end{lemma}

Using Lemma \ref{calc} along with equations (\ref{fnp1}), (\ref{fnp0}) and (\ref{gp0}), we have the following general result.

\begin{theorem}
\label{genres}
Let $K\subset \C^d$ be compact and nonpluripolar; let $\{K_n\}$ be any sequence of compact sets which decrease to $K$; i.e., $K_{n+1}\subset K_n$ for all $K$ and 
$K =\cap_n K_n$; and let $w$ be an admissible weight function on $K_n$ for $n\geq n_0$ whose restriction to $K$ is admissible. Let $\{\mu_n\}$ be a sequence of probability measures on $K_n$ with the property that 
\begin{equation}\label{fnhyp}
\lim_{n\to \infty}\frac{1}{2l_n}\log\,{\rm det}(G_n^{\mu_n,w})= \log(\delta^{w}(K))
\end{equation}
i.e., $\lim_{n\to \infty}f_n(0)=g(0)$. Then 
\begin{equation}\label{strongasym}
\frac{1}{m_n}B_n^{\mu_n,w} d\mu_n \to \mu_{K,Q}=(dd^cV_{K,Q}^*)^d \ \hbox{weak-}*.
\end{equation}
\end{theorem}

In particular, we get Theorem \ref{maint} on AAWF arrays.

\begin{theorem*}  Let $K\subset \C^d$ be compact and nonpluripolar, let $\{K_n\}$ be a sequence of compact sets which decrease to $K$; i.e., $K_{n+1}\subset K_n$ for all $K$ and 
$K =\cap_n K_n$, and let $w$ be an admissible weight function on $K_n$ for $n\geq n_0$ whose restriction to $K$ is admissible. For each $n$, take points $x_1^{(n)},x_2^{(n)},\cdots,x_{m_n}^{(n)}\in K_n$ for which 
\begin{equation}\label{wam}
 \lim_{n\to \infty}\bigl[|VDM(x_1^{(n)},\cdots,x_{m_n}^{(n)})|w(x_1^{(n)})^nw(x_2^{(n)})^n\cdots w(x_{m_n}^{(n)})^n\bigr]^{\frac{1}{l_n}}=\delta^w(K)
\end{equation}
and let $\mu_n:= \frac{1}{m_n}\sum_{j=1}^{m_n} \delta_{x_j^{(n)}}$. Then
$$
\mu_n \to \mu_{K,Q} \ \hbox{weak}-*.
$$
\end{theorem*}

\begin{proof} As observed before the proof of Lemma \ref{1stderiv}, we have $B_n^{\mu_n,w}(x_j^{(n)})=m_n$ for $j=1,...,m_n$ and hence a.e. $\mu_n$ on $K_n$. Hence the result follows immediately from Theorem \ref{genres}, specifically, equation (\ref{strongasym}). Alternately, if $\mu_n^t$ is an optimal measure of order $n$ for $K_n$ and $w_t$, we have 
$$\det G_n^{\mu_n,w_t} \leq \det G_n^{\mu_n^t,w_t}$$
and hence
$$\liminf f_n(t) \geq g(t) \ \hbox{for all} \ t;$$
finally, by hypothesis,
$$\lim_{n\to\infty}f_n(0)=-\log(\delta^{w}(K))=g(0).$$
Thus Lemma \ref{calc} is valid to show $\mu_n\to \mu_{K,Q}$ weak-*. Indeed, in this case, as observed earlier, the functions $f_n(t)$ are affine in $t$ so that $f_n''(t)=0$ is immediate and Lemma \ref{2ndderiv} is unnecessary.
\end{proof}


\end{document}